\pdfoutput=1
\RequirePackage{ifpdf}
\ifpdf 
\documentclass[pdftex]{sigma}
\else
\documentclass{sigma}
\fi

\usepackage{multicol,pdflscape,bm}

\numberwithin{equation}{section}

\newtheorem{Theorem}{Theorem}[section]
\newtheorem{Proposition}[Theorem]{Proposition}
 { \theoremstyle{definition}
\newtheorem{Remark}[Theorem]{Remark} }

\begin{document}

\allowdisplaybreaks

\newcommand{\arXivNumber}{1809.07421}

\renewcommand{\thefootnote}{}

\renewcommand{\PaperNumber}{007}

\FirstPageHeading

\ShortArticleName{Supersingular Elliptic Curves and Moonshine}

\ArticleName{Supersingular Elliptic Curves and Moonshine\footnote{This paper is a~contribution to the Special Issue on Moonshine and String Theory. The full collection is available at \href{https://www.emis.de/journals/SIGMA/moonshine.html}{https://www.emis.de/journals/SIGMA/moonshine.html}}}

\Author{Victor Manuel ARICHETA~$^{\dag\ddag}$}

\AuthorNameForHeading{V.M.~Aricheta}

\Address{$^\dag$~Department of Mathematics, Emory University, Atlanta, GA 30322, USA}
\EmailD{\href{mailto:variche@emory.edu}{variche@emory.edu}}
\Address{$^\ddag$~Institute of Mathematics, University of the Philippines,\\
\hphantom{$^\ddag$}~Diliman 1101, Quezon City, Philippines}
\EmailD{\href{mailto:vmaricheta@math.upd.edu.ph}{vmaricheta@math.upd.edu.ph}}

\ArticleDates{Received September 30, 2018, in final form January 19, 2019; Published online January 29, 2019}

\Abstract{We generalize a theorem of Ogg on supersingular $j$-invariants to supersingular elliptic curves with level. Ogg observed that the level one case yields a characterization of the primes dividing the order of the monster. We show that the corresponding analyses for higher levels give analogous characterizations of the primes dividing the orders of other sporadic simple groups (e.g., baby monster, Fischer's largest group). This situates Ogg's theorem in a broader setting. More generally, we characterize, in terms of supersingular elliptic curves with level, the primes arising as orders of Fricke elements in centralizer subgroups of the monster. We also present a connection between supersingular elliptic curves and umbral moonshine. Finally, we present a procedure for explicitly computing invariants of supersingular elliptic curves with level structure.}


\Classification{14H52; 11F06; 11F11; 11F22; 11F37; 20D08}

\renewcommand{\thefootnote}{\arabic{footnote}}
\setcounter{footnote}{0}

\section{Introduction and results}\label{S1}

\emph{Moonshine} refers to unexpected connections between disparate areas of mathematics (e.g., modular objects, sporadic groups) and physics (e.g., 2-dimensional conformal field theories). This paper deals with a theorem of Ogg and his subsequent observation, which may be regarded as the earliest occurrence of moonshine. In this paper we show that these constitute the first case of a broader phenomenon.

In his 1975 inaugural lecture at the Coll\`ege de France, Tits mentioned that the (then conjectural) \emph{monster} group $\mathbb{M}$, if it exists, is a simple sporadic group of order
\begin{gather*}
|\mathbb{M}| = 2^{46} \cdot 3^{20} \cdot 5^9 \cdot 7^6 \cdot 11^2 \cdot 13^3 \cdot 17 \cdot 19 \cdot 23 \cdot 29 \cdot 31 \cdot 41 \cdot 47 \cdot 59 \cdot 71.
\end{gather*}
Ogg, who was in the audience, recognized that the primes in this factorization are precisely the ones he had recently obtained geometrically from his work on supersingular $j$-invariants \cite{ogg1974automorphismes}. In his work he showed that the following statements O1 and O2 are equivalent.
\begin{enumerate}\itemsep=0pt
\item[(O1)] All the supersingular $j$-invariants in characteristic $p$ are in the prime field $\mathbb{F}_p$.
\item[(O2)] The genus of the modular curve $X_0^+(p)$ is zero.
\end{enumerate}
Here $X_0^+(p)$ is the quotient of $X_0(p)$ by its Fricke involution $w_p$. (See Section~\ref{S2.1} for definitions.) We refer to this equivalence in this paper as \emph{Ogg's theorem}. \emph{Ogg's observation} is that the primes satisfying either O1 or O2 are the ones dividing the order of the monster.\footnote{For a short note on Ogg's observation, see \cite{duncan2016jack}.}

A few years later, in 1979, Conway and Norton~-- encouraged by observations of McKay and Thompson~\cite{thompson1979some}~-- published their \emph{monstrous moonshine} conjecture~\cite{conway1979monstrous}. They postulated the existence of an infinite-dimensional representation of the monster, with a $\mathbb{Z}$-grading, whose graded dimension function is the modular $j$-invariant. Moreover, they gave precise predictions for the other graded trace functions. Similar to the $j$-function, the other graded trace functions were expected to be \emph{principal moduli} for genus zero quotients of the upper half-plane. A candidate for such a representation of the monster, called the \emph{moonshine module}, was constructed by Frenkel, Lepowsky and Meurman in 1984 \cite{frenkel1984natural}. Borcherds showed in 1992 that the graded trace functions of the moonshine module are indeed the principal moduli predicted by Conway and Norton, thereby proving the monstrous moonshine conjecture~\cite{borcherds1992monstrous}.

Ogg's observation is partly explained by monstrous moonshine. For each prime divisor $p$ of the order of the monster, there is an element of the monster whose graded trace function is the principal modulus for $X_0^+(p)$. The genus of $X_0^+(p)$ is then forced to be zero. In other words, the primes dividing the order of the monster are necessarily included in Ogg's list of primes, i.e, the set of primes satisfying either of the equivalent statements O1 and O2. This however does not explain why Ogg's list does not contain more primes. In this paper we show that both Ogg's theorem and Ogg's observation generalize naturally. In so doing we provide further evidence that Ogg's observation is more than just a coincidence.

It is useful, for purposes of generalization, to state Ogg's theorem in terms of modular curves. The modular curve $X_0(1)$, whose non-cuspidal points parametrize isomorphism classes of elliptic curves, has good reduction modulo any prime $p$ (cf.\ Section~\ref{S2.1}). The supersingular points of~$X_0(1)$ modulo $p$ correspond to supersingular elliptic curves in characteristic~$p$, and a supersingular point is defined over the prime field $\mathbb{F}_p$ if and only if the $j$-invariant of the corresponding elliptic curve is in~$\mathbb{F}_p$. Thus, we can restate Ogg's theorem as follows: the supersingular points of~$X_0(1)$ modulo $p$ are all defined over $\mathbb{F}_p$ if and only if the genus of $X_0^+(p)$ is zero.

The main idea of this paper is to consider elliptic curves with level structure by repla\-cing~$X_0(1)$ by more general modular curves $X$. We give a characterization for the primes $p$ that have the \emph{rationality property for~$X$}, by which we mean that all the supersingular points of~$X$ modulo~$p$ are defined over~$\mathbb{F}_p$. This analysis yields several consequences. First, as mentioned earlier, we find that Ogg's theorem generalizes naturally (Theorem~\ref{T1}). Second, we discover that Ogg's observation also generalizes naturally (Theorem~\ref{T2}). Third, we find that the primes that have the rationality property detect the existence of \emph{mock modular forms} with nonzero \emph{shadows} in \emph{umbral moonshine} (Theorem~\ref{T3}). (See Sections~\ref{S2.2} and~\ref{S2.3} for definitions.)

In this paper, we consider the quotient of $X_0(N)$ by Atkin--Lehner involutions $w_e, w_f, \ldots$ which we denote by $X_0(N){+}e,f,\ldots$. The non-cuspidal points of these modular curves represent isomorphism classes of elliptic curves with level structure. By a theorem of Igusa, these modular curves have good reductions modulo primes $p$ not dividing $N$. Let $X = X_0(N){+}e,f,\ldots$. Denote by: $Q_p(X)$ the number of supersingular points of $X$ modulo~$p$~-- or equivalently the number of isomorphism classes of supersingular elliptic curves with level structure in characteristic $p$~-- that are not defined over $\mathbb{F}_p$; $\operatorname{\textsl{genus}}(X)$ the genus of $X$; and $X^p$ the modular curve $X_0(Np){+}p,e,f,\ldots$ obtained by taking the quotient of $X_0(Np)$ by Atkin--Lehner involutions $w_p, w_e, w_f, \ldots$ where $\{ p, e, f, \ldots \}$ is understood to be the set $\{e,f,\ldots,p,pe,pf,\ldots \}$. Our first theorem gives a characterization for the primes that have the rationality property for these modular curves.

\begin{Theorem}[Theorem~\ref{theorem1}]\label{T1}Let $N$ be a positive integer, let $e,f,\ldots$ be exact divisors of $N$, and let $X = X_0(N){+}e,f,\ldots$. If $p$ is a prime not dividing $N$, then
\begin{gather*} \dfrac{1}{2}Q_p(X) = \operatorname{\textsl{genus}}(X^p) - \operatorname{\textsl{genus}}(X).\end{gather*} Consequently, $p$ has the rationality property for $X$ if and only if the modular curves $X$ and $X^p$ have the same genus.
\end{Theorem}

Theorem \ref{T1} naturally generalizes Ogg's theorem. Indeed, by letting $X = X_0(1)$, so that $\operatorname{\textsl{genus}}(X) = 0$ and $X^p = X_0^+(p)$, this theorem says that $p$ has the rationality property for $X_0(1)$ if and only if the genus of $X_0^+(p)$ is zero, which is precisely Ogg's theorem.

We now discuss a generalization of Ogg's observation. For this, we consider modular curves $X$ of the form $X_0(N){+}e,f,\ldots$ with genus equal to zero. All such modular curves except three (i.e., $X_0(25)$, $X_0(49){+}49$, $X_0(50){+}50$) arise in monstrous moonshine. By Theorem \ref{T1}, identifying the primes that have the rationality property for $X$ reduces to determining the complete list of primes $p$ such that $X^p$ has genus zero. We obtain the complete list of primes that have the rationality property for all such $X$, and we collect this information in a table in Appendix~\ref{AA}. Note from this table that there are no primes that have the rationality property for the three non-monstrous modular curves.

\begin{table}\centering
\caption{Primes that have the rationality property for low levels.}\vspace{1mm}
\label{Tab1}
\begin{tabular}{|c|c|c|}
\hline
$X$ & primes $p$ that have the rationality property for $X$ \\
\hline \hline
$X_0^+(2)$ & 3, 5, 7, 11, 13, 17, 19, 23, 31, 47 \bsep{1pt}\tsep{1pt}\\
\hline
$X_0(2)$ & 3, 5, 7, 11, 23 \\
\hline
$X_0^+(3)$ & 2, 5, 7, 11, 13, 17, 23, 29 \bsep{1pt}\tsep{1pt} \\
\hline
$X_0(3)$ & 2, 5, 11 \\
\hline
\end{tabular}
\end{table}

In Table \ref{Tab1}, we list the primes that have the rationality property for some modular curves. Notice that the primes on the first row~-- the primes that have the rationality property for~$X_0^+(2)$~-- are exactly the odd primes dividing the order of the \emph{baby monster} sporadic group. Similarly, the primes in the third row are the primes, not equal to~$3$, that divide the order of the largest \emph{Fischer group}, another sporadic group. Thus, the primes that have the rationality property for modular curves of higher level characterize the primes dividing the orders of other sporadic groups; these are natural generalizations of Ogg's observation.

\begin{Remark}These generalizations of Ogg's observation to the baby monster group and the largest Fischer group have been found, recently and independently, by Nakaya using analytic methods (e.g., class number estimates)~\cite{nakaya2018number}. He conjectured generalizations of Ogg's observation to the Harada--Norton and Held sporadic groups. Theorem~\ref{T1} (cf.\ entries for $X_0(5){+}$ and $X_0(7){+}$ in Appendix~\ref{AA}) confirms this conjecture.
\end{Remark}

In fact, we generalize Ogg's observation to all \emph{monstrous modular curves}~-- i.e.\ modular curves arising in monstrous moonshine~-- of the form $X_0(N){+}e,f,\ldots$ as follows. Under the correspondence given by monstrous moonshine, if $X \neq X_0(27){+}27$, then there is a unique conjugacy class of the monster whose graded trace function is the principal modulus for $X$, and we denote this conjugacy class by $C(X)$. If $X = X_0(27){+}27$, then there are exactly two conjugacy classes of the monster whose graded trace functions are both equal to the principal modulus for $X$. We let $C(X_0(27){+}27)$ be the smaller of these conjugacy classes; this class is labelled 27A in the ATLAS~\cite{conway1985atlas}. Finally, by a Fricke element of the monster of prime order~$p$, we mean an element of the monster whose graded trace function is the principal modulus for~$X_0^+(p)$. A~Fricke element of order $p$ is the same as any representative of the conjugacy class labelled $pA$ in the ATLAS.

\begin{Theorem}[Theorem~\ref{theorem2}]\label{T2}Let $X$ be a monstrous modular curve $X$ of the form $X_0(N){+}e,$ $f,\ldots$. If $p$ is a prime that does not divide $N$, then $p$ has the rationality property for $X$ if and only if the centralizer of $g \in C(X)$ in the monster contains a Fricke element of order $p$.
\end{Theorem}

Ogg's observation is again recovered from Theorem \ref{T2} by setting $X = X_0(1)$, and thus we find that Ogg's observation, just like Ogg's theorem, is the first case of a general phenomenon. Theorem \ref{T2} suggests that Ogg's observation is not a statement about the primes dividing the order of the monster per se, but a statement about the Fricke elements of the monster. This hints that a better understanding of the Fricke elements is necessary for a full understanding of Ogg's observation.

\begin{Remark}For the monstrous modular curves of the form $X = X_0(n|h){+}e,f,\ldots$ one can analogously define $X^p$ to be the modular curve $X_0(np|h){+}p,e,f,\ldots$ where as before $\{p,e,f,\ldots\}$ is the set $\{e,f,\ldots,p,pe,pf,\ldots\}$. In view of the characterization given by Theorem~\ref{T1}, we may say that a prime~$p$ has the rationality property for a monstrous modular curve $X$ if $X$ and $X^p$ have the same genus. With this generalized definition, Theorem \ref{T2} holds for any monstrous modular curve~$X$.
\end{Remark}

Theorem \ref{T2} gives a connection between supersingular elliptic curves with level structure and monstrous moonshine. We point out another connection, an unexpected one, between supersingular elliptic curves and umbral moonshine. (See Section~\ref{S2.3} for definitions.) Each of the 23 cases of moonshine in umbral moonshine has an associated genus zero modular curve called its \emph{lambency}. In this paper, we refer to these 23 modular curves as \emph{umbral modular curves}. For each umbral modular curve $X$, there is an \emph{umbral group} $G^X$ that serves as an analogue of the monster group, i.e., there exists a graded representation of $G^X$ such that the graded trace function $H_g^X$ of $g\in G^X$ is a distinguished vector-valued mock modular form of weight 1/2. There is also a naturally defined quotient of $G^X$ denoted $\bar{G}^X$, and we denote by $n_g$ the order of the image of $g\in G^X$ in $\bar{G}^X$. This integer $n_g$ is the level of $H_g^X$.

\begin{Theorem}[Theorem~\ref{theorem3}]\label{T3}Let $X$ be an umbral modular curve, and let $p$ be a prime not dividing the level of $X$. Then the prime $p$ has the rationality property for $X$ if and only if there exists an element $g \in G^X$ such that $n_g = p$ and the shadow of $H^X_g$ is nonzero.
\end{Theorem}

A mock modular form is a classical modular form if and only if its shadow is zero. Therefore, Theorem \ref{T3} says that the primes that have the rationality property for the umbral modular curves are exactly the primes that occur as levels of strictly mock modular graded trace functions.

Theorem \ref{T3} gives a connection between supersingular elliptic curves and umbral moonshine, while Theorem \ref{T2} provides another connection between supersingular elliptic curves and monstrous moonshine. We may then regard supersingular elliptic curves as a link between monstrous and umbral moonshine.

Lastly, we consider the modular curves of the form $X_0(N)$ of genus zero, and present another way of checking whether a prime has the rationality property for $X_0(N)$. This alternative method explicitly computes \emph{supersingular polynomials for $X_0(N)$}, which we will define shortly. The point is that a prime $p$ has the rationality property for $X_0(N)$ if and only if the $p$th supersingular polynomial for $X_0(N)$ splits completely into linear factors over $\mathbb{F}_p$. Note that these supersingular polynomials have already appeared in the literature, for example, in relation to the Kaneko--Zagier differential equations for low level Fricke groups~\cite{sakai2015modular}, and in connection to Atkin orthogonal polynomials \cite{sakai2011atkin, tsutsumi2007atkin}. However, the methods for computing these supersingular polynomials have been written down only for low levels. We give here a way for calculating these polynomials for all~$X_0(N)$ of genus zero.

Let $p$ be a prime not dividing $N$, and let $T_N$ be the principal modulus for $X_0(N)$ given in Appendix~\ref{AB}. Let $\text{SS}(N)$ be the set of supersingular points of $X_0(N)$ modulo $p$. The $p$th supersingular polynomial for $X_0(N)$ is the polynomial
\begin{gather*}
\text{ss}_p^{(N)}(x) := \prod_{E \in \text{SS}(N)} (x - T_N(E)).
\end{gather*}
That the rationality property can be determined by looking at the splitting property of this polynomial follows from the definition, because a point $x$ in $X_0(N)$ modulo $p$ is defined over $\mathbb{F}_p$ if and only if $T_N(x) \in \mathbb{F}_p$.

In Section~\ref{S4.1}, we describe how a certain polynomial $f_p^{(N)} \in \mathbb{Q}[x]$ arises from a modular form~$f$ of weight $p-1$ and level~1. If $f$ is chosen to be the weight $p-1$ Eisenstein series (among others), then this polynomial turns out to encode almost all the $T_N$-values of supersingular points on~$X_0(N)$ modulo~$p$. The few supersingular points not covered by this polynomial is encoded in another polynomial $g_p^{(N)}$ which we give in Appendix~\ref{AD}. (See also Section~\ref{S4.2}.)

\begin{Theorem}\label{T4}Let $p\geq 5$ be a prime, and let $f$ be any of $E_{p-1}$, $G_{p-1}$, or $H_{p-1}$ as defined in Section~{\rm \ref{S4}}. Then
\begin{gather*} \text{ss}_p^{(N)} \equiv \pm f_p^{(N)}g_p^{(N)} \pmod{p}. \end{gather*}
\end{Theorem}

It is straightforward to write an algorithm in a computer algebra system (e.g., Sage~\cite{developers2016sagemath}, PARI/GP \cite{paribordeaux}) that takes a modular form~$f$ of weight $p-1$ and level one as an input and produces the polynomial~$f^{(N)}_p$ as an output. Therefore, Theorem~\ref{T4} provides a simple way of explicitly computing supersingular polynomials for~$X_0(N)$. The first few supersingular polynomials for~$X_0(2)$ are given in Table~\ref{Tab2}. Note that these polynomials split completely into linear factors over $\mathbb{F}_p$ for $p = 5, 7, 11, 23$. Therefore these primes have the rationality property for~$X_0(2)$, a fact that agrees with the information in Table~\ref{Tab1}. Moreover, the number of quadratic factors in $\text{ss}_p^{(2)}$ coincides with the genus of $X_0(2)^p = X_0(2p){+}p$ (cf.\ Table~\ref{Tab3} in Section~\ref{S3.1}) which is consistent with Theorem~\ref{T1}.

\begin{table}\centering
\caption{Supersingular polynomials for $X_0(2)$.}\label{Tab2}\vspace{1mm}

\begin{tabular}{|c|c|}
\hline
$p$ & $\text{ss}^{(2)}_p(x)$\tsep{3pt}\\
\hline \hline
5 & $(x + 1)$\tsep{1pt}\bsep{1pt}\\
\hline
7 & $(x + 1) \cdot (x + 6)$\tsep{1pt}\bsep{1pt}\\
\hline
11 & $(x + 3) \cdot (x + 5) \cdot (x + 9)$ \tsep{1pt}\bsep{1pt}\\
\hline
13 & $(x + 1) \cdot \big(x^{2} + 8 x + 1\big)$\tsep{2pt}\bsep{1pt}\\
\hline
17 & $(x + 1) \cdot (x + 16) \cdot \big(x^{2} + 13 x + 16\big)$ \tsep{2pt}\bsep{1pt}\\
\hline
19 & $(x + 1) \cdot (x + 7) \cdot (x + 11) \cdot \big(x^{2} + 9 x + 11\big)$\tsep{2pt}\bsep{1pt} \\
\hline
23 & $(x + 3) \cdot (x + 5) \cdot (x + 15) \cdot (x + 16) \cdot (x + 17) \cdot (x + 18)$\tsep{1pt}\bsep{1pt}\\
\hline
29 & $(x + 16) \cdot (x + 23) \cdot (x + 24) \cdot \big(x^{2} + 24 x + 16\big) \cdot \big(x^{2} + 25 x + 23\big)$ \tsep{2pt}\bsep{1pt}\\
\hline
\end{tabular}
\end{table}

The rest of the paper is organized as follows. The necessary background on modular curves, mock modular forms and moonshine is given in Section~\ref{S2}. Theorems~\ref{T1}, \ref{T2} and \ref{T3} are proven in Section~\ref{S3}. Finally, the proof of Theorem~\ref{T4} is given in Section~\ref{S4}.

\section{Moduli spaces, mock modular forms, moonshine}\label{S2}

\subsection{Moduli spaces}\label{S2.1}

In this subsection, we define the modular curves $X_0(N)$ and the Atkin--Lehner involutions on them. We also describe the Deligne--Rapoport model for $X_0(pN)$ modulo a prime $p$ not divi\-ding~$N$, closely following Ogg's description in~\cite{ogg1975reduction}.

For every positive integer $N$, the congruence subgroup $\Gamma_0(N)$ acts on the complex upper half-plane $\mathbb{H}$ by linear fractional transformations. The orbit space $Y_0(N) := \Gamma_0(N)\backslash \mathbb{H}$ is naturally a Riemann surface, and it admits a moduli interpretation: the points of~$Y_0(N)$ parametrize isomorphism classes of cyclic isogenies, of degree~$N$, of complex elliptic curves. This Riemann surface can be compactified by adjoining the orbits of~$\Gamma_0(N)$ on $\mathbb{Q} \cup \{\infty\}$, called the cusps of~$\Gamma_0(N)$, and we denote this compactification by~$X_0(N)$.

Let $a,b,c,d \in \mathbb{Z}$ and let $e$ be an exact divisor of $N$, by which we mean that $e|N$ and \mbox{$(e,N/e) = 1$}. If the matrix $w_e:= (\begin{smallmatrix} ae & b \\ cN & de \end{smallmatrix})$ has determinant~$e$, then $\frac{1}{\sqrt{e}}w_e$ lies in the normalizer of $\Gamma_0(N)$ in $\text{SL}_2(\mathbb{R})$. This element of the normalizer induces an involution on $X_0(N)$, called an Atkin--Lehner involution of~$X_0(N)$, which is also denoted by $w_e$. (The Atkin--Lehner involution~$w_N$ of $X_0(N)$ is also known as the Fricke involution.) In order to describe the action of $w_e$ using the moduli interpretation of $X_0(N)$, separate an $N$-isogeny into sub-isogenies of degrees~$e$ and~$N/e$. The involution $w_e$ acts as the transpose on the $e$-part of the isogeny and leaves the $N/e$-part fixed.

By a theorem of Igusa, for any prime $p$ not dividing $N$, the modular curve $X_0(N)$ has good reduction modulo $p$. On the other hand, the reduction of $X_0(pN)$ modulo $p$ is a singular curve obtained from glueing two copies of $X_0(N)$ modulo $p$ at the supersingular points. More precisely, the non-cuspidal points of $X_0(pN)$ modulo $p$ parametrize $pN$-isogenies of elliptic cuves over $\mathbb{F}_p$. We separate a $pN$-isogeny into its $N$-part and its $p$-part. There are as many $N$-isogenies in characteristic~0 as in characteristic $p$, but there are only two $p$-isogenies in characteristic $p$, namely the Frobenius and its transpose. The first copy of $X_0(N)$ modulo $p$ parametrizes those isogenies whose $p$-part is the Frobenius; the second copy, those whose $p$-part is the transpose of the Frobenius. Their intersection consists of isogenies whose $p$-part may be thought of as either the Frobenius or its transpose~-- the supersingular points.

\subsection{Mock modular forms}\label{S2.2}

In this subsection, we recall the definition of a mock modular form and its shadow. These objects, which have their origins in the last letter of Ramanujan to Hardy, are now found in several areas of contemporary mathematics including moonshine. We refer to \cite{bringmann2017harmonic} for more details about mock modular forms and their applications.

Let $\Gamma$ be a discrete subgroup of $\text{SL}_2(\mathbb{R})$ and let $k$ be a half-integer. We say that a holomorphic function $f$ on $\mathbb{H}$ is a \emph{mock modular form of weight $k$ for $\Gamma$} if it has at most exponential growth as $\tau$ approaches any cusp of $\Gamma$, and if there exists a modular form $S(f)$ of weight $2-k$ on $\Gamma$ such that the sum $f + S(f)^*$ transforms like a holomorphic modular form of weight $k$ on $\Gamma$. Here, the function $S(f)^*$ is a solution to the differential equation:
\begin{gather*} (4\pi y)^k\dfrac{\partial S(f)^*(\tau)}{\partial \overline{\tau}} = -2\pi {\rm i} \overline{S(f)(\tau)}, \qquad \tau = x + {\rm i}y.\end{gather*}
The modular form $S(f)$ is called the \emph{shadow of $f$}, is uniquely determined by $f$, and is equal to zero if and only if $f$ is a usual modular form.

\subsection{Umbral moonshine}\label{S2.3}

Over the last 40 years, several cases of moonshine have been observed and proven. Most notable for the amount of research that their discovery inspired are the original monstrous moonshine, discussed in Section~\ref{S1}, and the more recent Mathieu moonshine. The latter is now known to belong to a family of moonshine collectively called umbral moonshine, which is the topic of this subsection.

In 2010 Eguchi, Ooguri and Tachikawa observed a numerical coincidence reminiscent of the McKay--Thompson observation \cite{eguchi2011notes}. The decomposition of the elliptic genus of a K3 surface into irreducible characters of the $\mathcal{N}=4$ superconformal algebra gives rise to a $q$-series
\begin{gather*} H(\tau) = 2q^{-1/8}\big({-}1 + 45q + 231q^2 + 770q^3 + 2277q^4 + \cdots\big).\end{gather*}
It was noted that $H(\tau)$ is a mock modular form and its first few coefficients are dimensions of irreducible representations of the largest Mathieu group $M_{24}$. Mathieu moonshine, formulated in a series of papers \cite{cheng2010k3, eguchi2011note, gaberdiel2010mathieu, gaberdiel2010mathieu2} and proven by Gannon \cite{gannon2016much}, is the statement that there exists an infinite-dimensional graded representation of $M_{24}$ with the following property: the graded trace functions are certain distinguished mock modular forms of weight~$1/2$. These graded trace functions are Rademacher sums which are natural weight 1/2 analogues of the principal modulus property \cite{cheng2012rademacher}. Mathieu moonshine thus expanded the class of automorphic objects considered in moonshine to include mock modular forms and other weights. In a series of papers Cheng, Duncan and Harvey identified Mathieu moonshine as one of a family of correspondences between finite groups and mock modular forms \cite{cheng2014umbral, cheng2014umbral2, cheng2018weight}. They referred to this conjectured family of correspondences as umbral\footnote{The word umbral was chosen to highlight the existence of \emph{shadows} in this {moon}shine.} moonshine.

Briefly, umbral moonshine is a collection of 23 cases of moonshine relating groups arising from lattices to (vector-valued) mock modular forms. The lattices in umbral moonshine are the \emph{Niemeier lattices}, which are the even unimodular self-dual lattices with roots (i.e.\ lattice vectors of length 2), and they are determined by their \emph{Niemeier root systems}. (For more information about root systems and their Dynkin diagrams, see \cite{humphreys2012introduction}.) These root systems have rank 24 and their simple components are root systems of ADE type of the same Coxeter number. A Niemeier root system is said to be of \emph{$A$-type} if it has a simple component of type~$A$, it is said to be of \emph{$D$-type} if it has a simple component of type $D$ but no type $A$ component, and it is said to be of \emph{$E$-type} if it only has type $E$ components.

Given a Niemeier root system, there is a group of genus zero naturally attached to it called its \emph{lambency}. The lambencies are defined as follows: The Coxeter number of a Niemeier lattice is the common Coxeter number of its simple components. The Coxeter numbers of the $A$-type Niemeier lattices are the same integers $N$ for which the genus of $\Gamma_0(N)$ is zero, and the lambency of such a Niemeier lattice of Coxeter number $N$ is defined to be $\Gamma_0(N)$. Similarly, the Coxeter numbers of the $D$-type Niemeier lattices are the same integers $2N$ for which the genus of $\Gamma_0(2N){+}N$ is zero, and the lambency of such a Niemeier lattice of Coxeter number $2N$ is defined to be $\Gamma_0(2N){+}N$. As discussed in~\cite{cheng2014umbral2}, the genus zero groups naturally attached to the Niemeier root systems $E_6^4$ and $E_8^3$ of $E$-type are $\Gamma_0(12){+}4$ and $\Gamma_0(30){+}6,10,15$ respectively, and these corresponding groups are defined to be their lambencies.

Let $X$ be a Niemeier root system and let $L^X$ be the associated Niemeier lattice. The reflections through the roots of $L^X$ generate a normal subgroup of the full automorphism group of $L^X$ known as the \emph{Weyl group} of $X$. The \emph{umbral group} $G^X$, which plays the same role as the monster in monstrous moonshine, is defined to be the quotient of the group of automorphisms of $L^X$ by the Weyl group of $X$. Umbral moonshine associates to each element $g\in G^X$ a distinguished (vector-valued) mock modular form $H^X_g$. Duncan, Griffin and Ono showed the existence of an infinite-dimensional graded representation of $G^X$ with graded trace functions equal to $H^X_g$ \cite{duncan2015proof}. There is also a naturally defined quotient of $G^X$ denoted $\bar{G}^X$, and we denote by $n_g$ the order of the image of $g\in G^X$ in $\bar{G}^X$. This integer $n_g$ is the level of $H_g^X$.

\section{Supersingular elliptic curves and moonshine}\label{S3}

\subsection{Higher level Ogg's theorem}\label{S3.1}

\looseness=-1 In this subsection we prove Theorem \ref{T1}, a generalization of Ogg's theorem, and use this to obtain the primes that have the rationality property for the curves $X_0(N){+}e,f,\ldots$ of genus zero.

Let $X = X_0(N){+}e,f,\ldots$ and let $p$ be a prime not dividing $N$. Recall from Section~\ref{S1} the following notations: $Q_p(X)$ is the number of supersingular points of $X$ modulo $p$ that are not defined over $\mathbb{F}_p$; $\operatorname{\textsl{genus}}(X)$ is the genus of $X$; and $X^p$ is the modular curve $X_0(Np){+}p,e,f,\ldots$ obtained by taking the quotient of $X_0(Np)$ by Atkin--Lehner involutions $w_p$, $w_e$, $w_f$, $\ldots$ where $\{p,e,f,\ldots\}$ is understood to be the set $\{ e,f,\ldots, p, pe, pf, \ldots \}$.

\begin{Theorem}\label{theorem1} Let $N$ be a positive integer, let $e,f,\ldots$ be exact divisors of $N$, and let $X = X_0(N){+}e,f,\ldots$. If $p$ is a prime not dividing $N$, then
\begin{gather*} \dfrac{1}{2}Q_p(X) = \operatorname{\textsl{genus}}(X^p) - \operatorname{\textsl{genus}}(X).\end{gather*} Consequently, $p$ has the rationality property for $X$ if and only if the modular curves $X$ and $X^p$ have the same genus.
\end{Theorem}

\begin{proof}The proof follows that of Ogg \cite{ogg1975reduction}. The ingredient that we need for his proof to go through is a model for the (singular) curve $X_0(Np){+}e,f,\ldots$ modulo $p$. We recall the model for this curve here which is explained in~\cite[Section~5]{furumoto1999hyperelliptic}. If $e,f,\ldots$ are exact divisors of~$N$, so that each of them is coprime with~$p$, and if
\begin{gather*} X := X_0(N){+}e,f,\ldots \mod p,\qquad
 X(p) := X_0(Np){+}e,f,\ldots \mod p,\end{gather*}
then $X(p)$ consists of two copies of $X$ that intersect at the supersingular points of $X$ modulo $p$.

\looseness=-1 From here, the proof of Ogg goes through: If $X_1$ and $X_2$ are the components of $X(p)$, then $w_p$ defines an isomorphism of $X_1$ onto $X_2$ that acts as the Frobenius on $X_1 \cap X_2$~-- the supersingular points of $X$. Therefore the model of $X(p)/(w_p) = X^p$ modulo $p$ is given by one copy of $X$ which intersects itself at each point corresponding to a pair of conjugate supersingular points of~$X$, i.e., the supersingular points of $X$ not defined over $\mathbb{F}_p$. From this model of $X^p$, we obtain $\operatorname{\textsl{genus}}(X^p) = \operatorname{\textsl{genus}}(X) + \frac{1}{2}Q_p(X)$, which is the first part of Theorem \ref{T1}. The second part of the theorem readily follows since by definition $p$ has the rationality property for $X$ if and only if $Q_p(X) = 0$.
\end{proof}

We apply Theorem \ref{T1} to the case when the genus of $X = X_0(N){+}e,f,\ldots$ is zero. Considering such cases leads to a generalization of Ogg's observation which is the subject of the next subsection. According to Theorem~\ref{T1}, given such an~$X$ the prime $p$ has the rationality property for~$X$ if and only if the genus of~$X^p$ is zero. For example, Table~\ref{Tab3} gives the genus of~$X^p$ for~$X$ of level~2 and for the first few primes. From this we see that the primes 3, 5, 7, 11, 13, 17, 19, 23, 31, 47 have the rationality property for~$X_0^+(2)$ and the primes 3, 5, 7, 11, 23 have the rationality property for~$X_0(2)$. Moreover we know from \cite{ferenbaugh1993genus} that this list is complete since 94 is the largest level of the form $2p$ for which there exists a genus zero quotient of $X_0(2p)$ by Atkin--Lehner involutions. Similarly, we enumerate all the primes $p$ for which $X^p$ has genus zero for all $X = X_0(N){+}e,f,\ldots$ of genus zero, and we compile the results in Appendix~\ref{AA}.

\begin{Remark}We identified the primes that have the rationality property for modular curves $X_0(N){+}e,f,\ldots$ of genus zero. We mention here that Ogg already obtained the primes that have the rationality property for modular curves $X_0(N)$ regardless of the genus \cite{ogg1975reduction}. He did not however consider their quotients by Atkin--Lehner involutions. His result is as follows: $2$ is the only prime that has the rationality property for $X_0(11)$ and $X_0(17)$; also, if $N\neq 11, 17$ and if the genus of $X_0(N)$ is positive, then there are no primes that have the rationality property for~$X_0(N)$. We point out here that there is a typo in his list for~$X_0(4)$; the prime $5$ should not be in the list.
\end{Remark}

\begin{table}[]\centering
\caption{Genus of $X^p$.}\label{Tab3}\vspace{1mm}

\begin{tabular}{|c|c|c|c|c|c|c|c|c|c|c|c|c|c|c|}
\hline
 $p$ & 3 & 5 & 7 & 11 & 13 & 17 & 19 & 23 & 29 & 31 & 37 & 41 & 43 & 47 \\ \hline
$\text{gen}\big(\big(X_0^+(2)\big)^p\big)$ & 0 & 0 & 0 & 0 & 0 & 0 & 0 & 0 & 1 & 0 & 1 & 1 & 1 & 0 \tsep{1pt}\bsep{1pt}\\ \hline
$\text{gen}\big(\big(X_0(2)\big)^p\big)$ & 0 & 0 & 0 & 0 & 1 & 1 & 1 & 0 & 2 & 1 & 4 & 3 & 4 & 1 \tsep{1pt}\bsep{1pt}\\ \hline
\end{tabular}
\end{table}

\subsection{Higher level Ogg's observation}\label{S3.2}

Suppose that $X = X_0(N){+}e,f,\ldots$ is a monstrous modular curve (cf.~Section~\ref{S1}). In this subsection we provide a characterization, which generalizes Ogg's observation, of the primes that have the rationality property for $X$.

Recall that in Section~\ref{S1}, we defined $C(X)$ to be the unique conjugacy class of the monster associated via monstrous moonshine to $X$ if $X \neq X_0(27){+}27$, and we defined $C(X_0(27){+}27)$ to be the conjugacy class of the monster labelled 27A in the ATLAS \cite{conway1985atlas}. We also defined a Fricke element of prime order $p$ to be an element of the monster whose graded trace function is the principal modulus for~$X^+_
0(p)$.

\begin{Theorem}\label{theorem2} Let $X$ be a monstrous modular curve $X$ of the form $X_0(N){+}e,f,\ldots$. If $p$ is a prime that does not divide $N$, then $p$ has the rationality property for $X$ if and only if the centralizer of $g \in C(X)$ in the monster contains a Fricke element of order~$p$.
\end{Theorem}

\begin{proof}Given a conjugacy class $C$ of a finite group $G$, denote by $C^n$ the conjugacy class of~$G$ containing the $n$th powers of elements of $C$. Suppose $C_1$ and $C_2$ are conjugacy classes of orders~$n_1$ and~$n_2$ such that $(n_1,n_2)=1$. Then the classes $C_1$ and $C_2$ have representatives that commute if and only if there exists a conjugacy class~$C$ of order~$n_1n_2$ such that $C^{n_2} = C_1$ and $C^{n_1} = C_2$.

Now, the centralizer of $g \in C(X)$ in the monster contains a Fricke element of order $p$ if and only if the conjugacy classes $C(X)$ and $pA$ have representatives that commute. From the previous paragraph, this occurs if and only if there exists a conjugacy class of order $pN$ whose $N$th power is $pA$ and whose $p$th power is $C(X)$. One can manually check from the power maps of the conjugacy classes of the monster, using GAP \cite{GAP4} for instance, that this latter condition occurs if and only if $p$ has the rationality property for $X$.
\end{proof}

We illustrate by way of an example how this proof works. We consider the case when $N = 2$. This implies that $X = X_0^+(2)$ or $X_0(2)$, and the conjugacy class $C(X)$ is $2A$ or $2B$ respectively. In Table~\ref{Tab4}, we label the columns by the primes $p$ dividing the order of the monster~-- these are the primes that can occur as prime orders of Fricke elements~-- and we label the rows by $C(X)$. The entries in row $C(X)$ and column $p$ are the conjugacy classes of the monster of order $2p$ whose $p$th power is $C(X)$. For each cell, we can check which of those conjugacy classes have 2nd power equal to $pA$, and write these in boldface. Then from the proof, there is a conjugacy class in row $C(X)$ and column $p$ in boldface if and only if the centralizer of $g \in C(X)$ contains a Fricke element of order $p$.

\begin{table}[]\centering
\caption{}\label{Tab4}\vspace{1mm}

\begin{tabular}{|c|c|c|c|c|c|c|c|c|}
\hline
 & 3 & 5 & 7 & 11 & 13 & 17 & 19 \\ \hline
$2A$ & \bm{$6A$} $6D$ & \bm{$10A$} $10C$ & \bm{$14A$} & \bm{$22A$} & \bm{$26A$} & \bm{$34A$} & \bm{$38A$} \\ \hline
$2B$ & $6B$ \bm{$6C$} $6E$ $6F$ & \bm{$10B$} $10D$ $10E$ & \bm{$14B$} $14C$ & \bm{$22B$} & $26B$ & & \\ \hline
\end{tabular}

\begin{tabular}{|c|c|c|c|c|c|c|c|}
\hline
 & 23 & 29 & 31 & 41 & 47 & 59 & 71\\ \hline
$2A$ & \bm{$46CD$} & & \bm{$62AB$} & & \bm{$94AB$} & & \\ \hline
$2B$ & \bm{$46AB$} & & & & & & \\ \hline
\end{tabular}
\end{table}

From this table, we find that the odd primes arising as order of Fricke elements in the centralizer of $g \in C(X_0^+(2))$ in the monster are 3, 5, 7, 11, 13, 17, 19, 23, 31, 47. These are the same primes that have the rationality property for $X_0^+(2)$. Similarly, the Fricke elements of odd prime orders in the centralizer of $g \in C(X_0(2))$ in the monster have orders 3, 5, 7, 11, 23, and these coincide with the primes that have the rationality property for $X_0(2)$.

\subsection{Observation related to umbral moonshine}\label{S3.3}

In this subsection, we consider the modular curves that occur (as lambencies) in the theory of umbral moonshine. Given such an umbral modular curve $X$, we present another characterization of the primes that have the rationality property for $X$. This will be in terms of the modularity properties of certain graded trace functions.

Recall from Section~\ref{S2.3} that associated to each of the 23 cases of umbral moonshine are: a~modular curve $X$ of genus zero called its lambency; an umbral group $G^X$; and a set of graded trace functions $H^X_g$ for each $g \in G^X$. The functions $H^X_g$ are vector-valued mock modular forms of weight~1/2 and level~$n_g$, and a formula for the shadow of $H^X_g$ may be given in terms of naturally defined characters of $G^X$ called \emph{twisted Euler characters} (cf.\ Section~5.1 of~\cite{cheng2014umbral2}). The formulas show that the shadow of $H_g^X$ is zero if and only if the value of all the twisted Euler characters at $g$ is 0.

Consider the case of umbral moonshine of lambency $X_0(2)$ (i.e., Mathieu moonshine). This has umbral group equal to the largest Mathieu group $M_{24}$. The following table gives the conjugacy classes $[g]$ of $M_{24}$ with $n_g$ an odd prime, and the values of the twisted Euler character $\overline{\chi}_g^A$ at these conjugacy classes.

\begin{center}
\begin{tabular}{|c|c|c|c|c|c|c|}
\hline
 $[g]$ & $3A$ & $3B$ & $5A$ & $7AB$ & $11A$ & $23AB$ \\ \hline
$n_g$ & 3 & 3 & 5 & 7 & 11 & 23 \\ \hline
$\overline{\chi}_g^A$ & 6 & 0 & 4 & 3 & 2 & 1 \\ \hline
\end{tabular}
\end{center}

From this table, we see that there exists an element $g$ of $G^X$, where $n_g$ is an odd prime, for which the shadow $H_g^X$ is \emph{nonzero} if and only if $p = 3, 5, 7, 11, 23$. These primes are precisely the ones that have the rationality property for $X_0(2)$. In this example, one could argue that these primes are also simply the primes appearing as~$n_g$, but as the next example shows, in some cases there are primes that appear as $n_g$ but not as the level of a mock modular form with non-vanishing shadow.

Consider the umbral moonshine case of lambency $X = X_0(5)$. The following table gives the conjugacy classes $[g]$ of $G^X$ whose order $n_g$ is a prime $p\neq 5$, and the values of the twisted Euler characters $\overline{\chi}_g^A$ and $\chi_g^A$ at these conjugacy classes.

\begin{center}
\begin{tabular}{|c|c|c|c|c|}
\hline
 $[g]$ & $2B$ & $2C$ & $3A$ & $6A$ \\ \hline
$n_g$ & 2 & 2 & 3 & 3 \\ \hline
$\overline{\chi}_g^A$ & 2 & 2 & 0 & 0 \\ \hline
$\chi_g^A$ & -2 & 2 & 0 & 0 \\ \hline
\end{tabular}
\end{center}

In this case of umbral moonshine, the graded trace function $H_g^X$ for any order 3 element $g$ of $G^X$ has shadow equal to zero, i.e., $H_g^X$ is a classical modular form. There is an element $g$, of prime order $p \neq 5$, of $G^X$ for which the shadow of $H_g^X$ is nonzero if and only if $p = 2$. The prime 2 happens to be the only prime that has the rationality property for $X_0(5)$.

In fact, this pattern persists and we have the following theorem.

\begin{Theorem}\label{theorem3} Let $X$ be an umbral modular curve, and let $p$ be a prime not dividing the level of~$X$. Then the prime $p$ has the rationality property for $X$ if and only if there exists an element $g \in G^X$ such that $n_g = p$ and the shadow of $H^X_g$ is nonzero.
\end{Theorem}

\begin{proof}From the tables of values of the twisted Euler characters in \cite{cheng2014umbral2}, one can enumerate the primes $p$ not dividing $N$ with the following properties: (1) there is a $g\in G^X$ with $n_g = p$; and (2) there is a twisted Euler character that does not vanish at $g$, or equivalently, the shadow of~$H_g^X$ is nonzero. By inspection, the primes that satisfy these properties are the same primes that have the rationality property for $X$.
\end{proof}

\begin{Remark}
We considered only the primes that do not divide the level in the formulation of our notion of rationality. It would be interesting to extend this notion so as to include all primes. Can this be done in such a way that Theorems \ref{T2} and \ref{T3} also generalize?
\end{Remark}

\begin{Remark}
As explained in a recent work of Cheng and Duncan, naturally attached to any lambency $X$~-- i.e., a genus zero quotient of $X_0(N)$ by a set of Atkin--Lehner involutions that does not include the Fricke involution~-- is an ``optimal mock Jacobi form'' $\phi^X$ of level 1 with integer coefficients \cite{cheng2016optimal}. These optimal mock Jacobi forms allow the recovery of the graded trace functions in umbral moonshine; if $X$ is an umbral lambency, then the components of $H^X_e$ are the coefficients in the theta decomposition of $\phi^X$, and the components of the other graded trace functions $H_g^X$ may be obtained, via certain multiplicative relations (cf.\ Tables~8 and~9 of~\cite{cheng2014umbral2}), from $\phi^{X'}$ where $X'$ is of lower lambency. There are extra lambencies in \cite{cheng2016optimal} that do not occur in umbral moonshine. A natural guess is that there is a generalization of umbral moonshine that incorporates these more general lambencies, and Theorem~\ref{T3} could serve as a consistency check for this generalization.
\end{Remark}

\section{Generalized supersingular polynomials}\label{S4}

In this section we prove Theorem \ref{T4} mentioned in the introduction. For an even positive integer~$k$, we denote by: $E_k$ the normalized Eisenstein series of weight~$k$; $G_k$ the coefficient of $X^k$ in $\big(1-3E_4(\tau)X^4 + 2E_6(\tau)X^6\big)^{-1/2}$; and $H_k$ the coefficient of $X^k$ in $\big(1-3E_4(\tau)X^4 + 2E_6(\tau)X^6\big)^{k/2}$.

\subsection[The polynomial $f_p^{(N)}$]{The polynomial $\boldsymbol{f_p^{(N)}}$}\label{S4.1}

Suppose that the genus of $X_0(N)$ is zero, and let $T_N$ be the principal modulus for $X_0(N)$. There is a normalized modular form $\Delta_N \in M_{12}(\Gamma_0(N))$, whose formula is given in Appendix~\ref{AB}, that vanishes only at the infinite cusp of $\Gamma_0(N)$ and nowhere else. Since $\Delta$ is a modular form of weight~12 and level~1 that vanishes only at the infinite cusp and nowhere else, the modular form~$\Delta_N$ may be considered as higher level analogues of $\Delta$, and hence the choice for its notation.

Let $p$ be prime. Note that we can uniquely write $p-1$ in the form
\begin{gather*} p-1 = 12m + 4\delta + 6\epsilon,\qquad \text{where}\quad m \geq \mathbb{Z}_{\geq 0},\quad \delta, \epsilon \in \{0,1\}. \end{gather*} Using the classical valence formula, if $f \in M_{p-1}(\Gamma_0(1))$, then $f/\big(E_4^\delta E_6^\epsilon\big) \in M_{12m}(\Gamma_0(1))$. We get a~modular function on $\Gamma_0(N)$ by dividing $f/\big(E_4^\delta E_6^\epsilon\big)$ by~$\Delta_N^m$. Moreover, since $\Delta_N$ vanishes only at $\infty$, the poles of $f/\big(E_4^\delta E_6^\epsilon \Delta_N^{m}\big)$ are supported at $\infty$. Thus there exists a polynomial $f_p^{(N)} \in \mathbb{C}[x]$ such that
\begin{gather*} \dfrac{f}{E_4^\delta E_6^\epsilon \Delta_N^m} = f_p^{(N)}(T_N). \end{gather*}
Note that if $f$ has integral Fourier coefficients, which is true for the modular forms $E_{p-1}$, $G_{p-1}$, $H_{p-1}$, then $f_p^{(N)}$ has rational coefficients.

\subsection[The polynomial $g_p^{(N)}$]{The polynomial $\boldsymbol{g_p^{(N)}}$}\label{S4.2}

Later in proving Theorem \ref{T4}, we will be using the following result~-- which is the $N=1$ case of Theorem~\ref{T4}~-- due to Deuring, Hasse, Deligne, Kaneko and Zagier~\cite{kaneko1998supersingular}.

\begin{Proposition}\label{P1}Let $p \geq 5$ be a prime and let $f$ be any of $E_{p-1}$, $G_{p-1}$ or $H_{p-1}$. Let $f_p^{(1)}$ be the polynomial defined in Section~{\rm \ref{S4.1}}. Then
\begin{gather*} ss_p^{(1)}(x) \equiv \pm f_p^{(1)}(x)x^\delta(x-1728)^\epsilon \pmod{p}. \end{gather*}
\end{Proposition}

We denote by $g_p^{(N)}$ the higher level analogues of the factor $x^\delta$ and $(x-1728)^{\epsilon}$ found in this proposition which we obtain as follows. In this proposition, the factor $x$ corresponds to the (isomorphism class of the) elliptic curve $y^2 = x^3 +1$ with $j$-invariant equal to~0. The exponent~$\delta$ is equal to 1 when $p \equiv 2 \pmod 3$. The proposition says that these are precisely the primes~$p$ for which the elliptic curve $y^2 = x^3 +1$ is supersingular in characteristic $p$. Similarly, the factor $x-1728$ corresponds to the (class of) curve $y^2 = x^3 + x$ with $j$-invariant 1728, and the proposition says that this elliptic curve is supersingular when $\epsilon = 1$ or when $p\equiv 3 \pmod 4$. 	

These distinguished isomorphism classes of elliptic curves with $j$-invariants 0 and 1728~-- or equivalently points on the moduli space $X_0(1)$ with $j$-values 0 or 1728~-- break up into several isomorphism classes when we consider them as elliptic curves with level structure $N$, or as points on the moduli space $X_0(N)$. The $T_N$-values of these points constitute the roots of the polynomial analogue of $x^\delta(x-1728)^\epsilon$ that we seek.

To obtain the $T_N$-values given a $j$-value, we need a \emph{modular relation} between $j$ and $T_N$, by which we mean a relation $j(\tau) = r_N(T_N(\tau))$ for some rational function $r_N \in \mathbb{Q}(x)$. In Appendix~\ref{AE}, we present the complete list of modular relations for $j$ and $T_N$ for the $N$'s such that~$X_0(N)$ has genus zero. One can verify these identities by checking that the Fourier coefficients for the left and the right hand side coincide up to the Sturm bound.

To find the $T_N$-values of the points with $j$-invariants 0 and 1728, we need only to solve the equations $0 = r_N(T_N)$ and $1728 = r_N(T_N)$. For example, for $N = 2$, we have $r_2(T) = (T+256)^3/T^2$ and so:
\begin{gather*} r_2(T) = 0 \ \Rightarrow \ T+256 = 0, \qquad
 r_2(T) = 1728 \ \Rightarrow \ (T-512)(T+64) = 0. \end{gather*}
Therefore the level 2 analogue of $x^\delta(x-1728)^\epsilon$ is the polynomial $(x+256)^\delta(x-512)^\epsilon(x+64)^\epsilon$. One can do the same for the other levels to obtain all the polynomials $g_p^{(N)}$.

\subsection{Proof of Theorem~\ref{T4}}\label{S4.3}

Note that the supersingular points of $X_0(N)$ modulo $p$ are the points lying above the supersingular points of $X_0(1)$ modulo $p$. Therefore, Proposition \ref{P1} tells us that the roots of the equation $f_p^{(1)}(j)j^\delta(j-1728)^\epsilon = 0$ are the $j$-values of the supersingular points of $X_0(N)$ modulo $p$. From the definition of $f^{(N)}_p$, we have the relation
\begin{equation}
f^{(N)}_p(T_N)\Delta^m_N = f^{(1)}_p(j)\Delta^m.
\end{equation}
Therefore $f_p^{(1)}(j) = 0$ if and only if $f_p^{(N)}(T_N) = 0$. Also by definition of $g_p^{(N)}$, the equation $j^\delta(j-1728)^\epsilon = 0$ if and only if $g_p^{(N)}(T_N) = 0$. Therefore, the roots of $g_p^{(N)}(T_N)f_p^{(N)}(T_N) = 0$ are the $T_N$-values of the supersingular points of $X_0(N)$ modulo $p$. Finally, the coefficient of $g_p^{(N)}(T_N)f_p^{(N)}(T_N)$ is $\pm 1$ because: $g_p^{(N)}$ is monic; and by (1) the leading term of $f_p^{(N)}$ is the same as the leading term of $f_p^{(1)}$, which is $\pm 1$ by Proposition \ref{P1}. \hfill $\blacksquare$

\appendix

\section{Primes that have the rationality property}\label{AA}

In the following table, we use the notation $N{+}e,f,\ldots$ for the modular curve $X_0(N){+}e,f,\ldots$. Moreover, we use the notation $N{+}$ when all the exact divisors of $N$ are included, and the notation $N{-}$ when no exact divisors are included. The following table lists the primes that have the rationality property for genus zero modular curves of the form $N{+}e,f,\ldots$. There are no primes that have the rationality property for the genus zero modular curves of the form $N{+}e,f,\ldots$ not listed in this table.

\begin{center}
\begin{tabular}{|l|c|c|}
\hline
\ \ $X$ &\multicolumn{1}{|p{50mm}|}{\ \ \ \ $\begin{array}{c} \text{primes that have} \\[-2pt] \text{the rat.\ prop.\ for $X$}\end{array}$} \\
\hline \hline
$2{+}$ & 3, 5, 7, 11, 13, 17, 19, 23, 31, 47 \\
\hline
$2-$ & 3, 5, 7, 11, 23 \\
\hline
$3{+}$ & 2, 5, 7, 11, 13, 17, 23, 29 \\
\hline
$3-$ & 2, 5, 11 \\
\hline
$4{+}$ & 3, 5, 7, 11, 23\\
\hline
$4{-}$ & 3, 7 \\
\hline
$5{+}$ & 2, 3, 7, 11, 19\\
\hline
$5{-}$ & 2\\
\hline
$6{+}$ & 5, 7, 11, 13\\
\hline
$6{+}6$ & 5, 11\\
\hline
$6{+}3$ & 5\\
\hline
$7{+}$ & 2, 3, 5, 17\\
\hline
$7{-}$ & 3\\
\hline
$8{+}$ & 3, 7\\
\hline
$9{+}$ & 2, 5\\
\hline
$9{-}$ & 2\\
\hline
$10{+}$ & 3, 7, 11\\
\hline
$10{+}5$ & 3\\
\hline
$11{+}$ & 2, 3, 5\\
\hline
$12{+}$ & 5\\
\hline
\end{tabular}
\qquad
\begin{tabular}{|l|c|c|}
\hline
\ \ $X$ & \multicolumn{1}{|p{50mm}|}{\ \ \ \ $\begin{array}{c} \text{primes that have} \\[-2pt] \text{the rat.\ prop.\ for $X$}\end{array}$} \\
\hline \hline
$13{+}$ & 2, 3\\
\hline
$14{+}$ & 3, 5\\
\hline
$14{+}14$ & 3\\
\hline
$15{+}$ & 2, 7\\
\hline
$15{+}15$ & 2\\
\hline
$17{+}$ & 2, 3, 7\\
\hline
$19{+}$ & 2, 5\\
\hline
$20{+}$ & 3\\
\hline
$21{+}$ & 2, 5\\
\hline
$22{+}$ & 3, 5\\
\hline
$23{+}$ & 2, 3\\
\hline
$25{+}$ & 2\\
\hline
$26{+}$ & 3\\
\hline
$27{+}$ & 2\\
\hline
$29{+}$ & 3\\
\hline
$31{+}$ & 2\\
\hline
$33{+}$ & 2\\
\hline
$35{+}$ & 2, 3\\
\hline
$47{+}$ & 2\\
\hline
$55{+}$ & 2\\
\hline
\end{tabular}
\end{center}

\vspace{-5.5mm}

\section[Principal moduli for $X_0(N)$]{Principal moduli for $\boldsymbol{X_0(N)}$}\label{AB}

Let $\eta(\tau)$ be the Dedekind eta function. In this table, we employ the following notation for an eta-product:
$n_1^{d_1}\cdots n_l^{d_l} := \eta(n_1\tau)^{d_1}\cdots \eta(n_l\tau)^{d_l}$.
The following table shows the principal moduli for the modular curves $X_0(N)$.
\vspace{-2mm}

\begin{center}
\begin{tabular}{|c|c|}
\hline
$N$ & $T_N$\\
\hline \hline
2 & $1^{24}/2^{24}$\\
\hline
3 & $1^{12}/3^{12}$ \\
\hline
4 & $1^8/4^8$\\
\hline
5 & $1^6/5^6$\\
\hline
6 & $2^8 3^4/1^4 6^8$\\
\hline
7 & $1^4/7^4$\\
\hline
8 & $1^44^2/2^28^4$ \\
\hline
\end{tabular}
\qquad
\begin{tabular}{|c|c|}
\hline
$N$ & $T_N$\\
\hline \hline
9 & $1^3/9^3$\\
\hline
10 & $2^45^2/1^210^4$\\
\hline
12 & $3^34^1/1^112^3$\\
\hline
13 & $1^2/13^2$\\
\hline
16 & $1^28^1/2^116^2$\\
\hline
18 & $2^29^1/1^118^2$\\
\hline
25 & $1/25$\\
\hline
\end{tabular}
\end{center}

\vspace{-4mm}

\section[Higher level analogues of $\Delta$]{Higher level analogues of $\boldsymbol{\Delta}$}\label{AC}

We again use the following notation for an eta-product:
$n_1^{d_1}\cdots n_l^{d_l} := \eta(n_1\tau)^{d_1}\cdots \eta(n_l\tau)^{d_l}$.
The following table lists the modular (non-cuspidal) forms $\Delta^{(N)} \in M_{12}(\Gamma_0(N))$ that vanish only at the cusp $\infty$ of $\Gamma_0(N)$ and nowhere else.

\vspace{-2mm}

\begin{center}
\begin{tabular}{|c|c|	}
\hline
$N$ & $\Delta^{(N)}$\\
\hline \hline
2 & $2^{48}/1^{24}$ \\
\hline
3 & $3^{36}/1^{12}$ \\
\hline
4 & $4^{48}/2^{24}$ \\
\hline
5 & $5^{30}/1^6$ \\
\hline
6 & $1^{12}6^{72}/2^{24}3^{36}$ \\
\hline
7 & $7^{28}/1^4$ \\
\hline
8 & $8^{48}/4^{24}$ \\
\hline
\end{tabular}
\qquad
\begin{tabular}{|c|c|c|	}
\hline
$N$ & $\Delta^{(N)}$\\
\hline \hline
9 & $9^{36}/3^{12}$\\
\hline
10 & $1^{6} 10^{60}/ 2^{12} 5^{30}$\\
\hline
12 & $2^{12} 12^{72}/ 4^{24} 6^{36}$\\
\hline
13 & $13^{26}/1^2$ \\
\hline
16 & $16^{48} / 8^{24}$ \\
\hline
18 & $3^{12} 18^{72}/6^{24} 9^{36}$ \\
\hline
25 & $25^{30}/5^6$ \\
\hline
\end{tabular}
\end{center}

\begin{landscape}

\section[The polynomials $g_p^{(N)}$]{The polynomials $\boldsymbol{g_p^{(N)}}$}\label{AD}

In the following table, we list the polynomials $g_p^{(N)}$ whose roots are the $T_N$-invariants of the characteristic $p$ supersingular elliptic curves with level $N$ structure, such that the $j$-invariant is 0 or 1728.

\begin{center}
\begin{tabular}{|c|c|}
\hline
$N$ & $g_p^{(N)}$ where $p-1 = 12m + 4\delta + 6\epsilon$ \tsep{3pt}\bsep{1pt}\\
\hline \hline
2 & $(x+256)^\delta(x-512)^\epsilon(x+64)^\epsilon$\\
\hline
3 & $(x+27)^\delta(x+243)^\delta(x^2-486x-19683)^\epsilon$ \\
\hline
4 & $(x^2 + 256x + 4096)^\delta(x + 32)^\epsilon (x^2 - 512x - 8192)^\epsilon$\\
\hline
5 & $(x^2 + 250x + 3125)^\delta(x^2 - 500x - 15625)^\epsilon(x^2+22x+125)^\epsilon$\\
\hline
6 & $(x + 3)^\delta(x^3 + 225x^2 - 405x + 243)^\delta(x^2 + 18x - 27)^\epsilon(x^4 - 540x^3 + 270x^2 - 972x + 729)^\epsilon$\\
\hline
7 & $(x^2+13x+49)^\delta(x^2 + 245x + 2401)^\delta(x^4 - 490x^3 - 21609x^2 - 235298x - 823543)^\epsilon$\\
\hline
8 & $(x^4 + 256x^3 + 5120x^2 + 32768x + 65536)^\delta(x^2 + 32x + 128)^\epsilon(x^4 - 512x^3 - 10240x^2 - 65536x - 131072)^\epsilon$ \\
\hline
9 & $(x + 9)^\delta(x^3 + 243x^2 + 2187x + 6561)^\delta(x^6 - 486x^5 - 24057x^4 - 367416x^3 - 2657205x^2 - 9565938x - 14348907)^\epsilon$ \\
\hline
10& \shortstack{$(x^6 + 230x^5 + 275x^4 - 1500x^3 + 4375x^2 - 6250x + 3125)^\delta(x^2 + 2x + 5)^\epsilon(x^2 + 20x - 25)^\epsilon$\\$(x^4 - 540x^3 + 1350x^2 - 1500x + 625)^\epsilon(x^2-2x+5)^\epsilon$}\\
\hline
12 & \shortstack{$(x^2 + 4x - 8)^\delta(x^6 + 228x^5 - 408x^4 - 128x^3 - 192x^2 + 768x - 512)^\delta$\\$(x^4 + 20x^3 - 48x^2 + 32x - 32)^\epsilon(x^8 - 536x^7 - 272x^6 + 3328x^5 + 6400x^4 - 20480x^3 + 4096x^2 + 16384x - 8192)^\epsilon$}\\
\hline
13 & \shortstack{$(x^2+5x+13)^\delta(x^4 + 247x^3 + 3380x^2 + 15379x + 28561)^\delta$\\$(x^6 - 494x^5 - 20618x^4 - 237276x^3 - 1313806x^2 - 3712930x - 4826809)^\epsilon(x^2+6x+13)^\epsilon$}\\
\hline
16 & \shortstack{$(x^8 + 256x^7 + 5632x^6 + 53248x^5 + 282624x^4 + 917504x^3 + 1835008x^2 + 2097152x + 1048576)^\delta(x^4 + 32x^3 + 192x^2 + 512x + 512)^\epsilon$\\$(x^8 - 512x^7 - 11264x^6 - 106496x^5 - 565248x^4 - 1835008x^3 - 3670016x^2 - 4194304x - 2097152)^\epsilon$} \\
\hline
18 & \shortstack{$ (x^3 + 3x^2 - 9x + 9)^\delta(x^9 + 225x^8 - 1080x^7 + 3348x^6 - 8262x^5 + 16038x^4 - 23328x^3 + 26244x^2 - 19683x + 6561)^\delta$\\$(x^6 + 18x^5 - 81x^4 + 216x^3 - 405x^2 + 486x - 243)^\epsilon$\\$(x^{12} - 540x^{11} + 1890x^{10} - 4212x^9 + 13527x^8 - 48600x^7 + 129276x^6 - 262440x^5 $\\ $+ 413343x^4 - 498636x^3 + 433026x^2 - 236196x + 59049)^\epsilon$}\\
\hline
25 & \shortstack{$(x^{10} + 250x^9 + 4375x^8 + 35000x^7 + 178125x^6 + 631250x^5 + 1640625x^4 + 3125000x^3 + 4296875x^2 + 3906250x + 1953125)^\delta$\\$(x^4 + 10x^3 + 45x^2 + 100x + 125)^\epsilon(x^2+2x+5)^\epsilon$\\$(x^{10} - 500x^9 - 18125x^8 - 163750x^7 - 871875x^6 - 3137500x^5 - 8203125x^4 - 15625000x^3 - 21484375x^2 - 19531250x - 9765625)^\epsilon$}\\
\hline
\end{tabular}
\end{center}
\vspace*{\fill}

\end{landscape}

\newpage

\begin{landscape}

\section{Modular relations}\label{AE}

The following table gives $j$ as rational functions of $T_N$.

\begin{center}
\begin{tabular}{|c|c|c|}
\hline
$N$ & modular relations between $j$ and $T = T_N$\\
\hline \hline
2 & $j = \dfrac{(T+256)^3}{T^2}$ \\
\hline
3 & $j = \dfrac{(T+27)(T+243)^3}{T^3}$ \\
\hline
4 & $j = \dfrac{(T^2+256T+4096)^3}{(T+16)T^4}$ \\
\hline
5 & $j = \dfrac{(T^2+250T+3125)^3}{T^5}$ \\
\hline
6 & $j = \dfrac{(T + 3)^3(T^3 + 225T^2 - 405T + 243)^3}{(T-1)^2(T-9)^6T^3}$ \\
\hline
7 & $j = \dfrac{(T^2 + 13T + 49)(T^2 + 245T + 2401)^3}{T^7}$\\
\hline
8 & $j = \dfrac{(T^4 + 256T^3 + 5120T^2 + 32768T + 65536)^3}{(T+4)(T+8)^2T^8}$ \\
\hline
9 & $j = \dfrac{(T+9)^3(T^3 + 243T^2 + 2187T + 6561)^3}{(T^2+9T+27)T^9}$\\
\hline
10 & $j = \dfrac{(T^6 + 230T^5 + 275T^4 - 1500T^3 + 4375T^2 - 6250T + 3125)^3}{(T - 1)^2(T - 5)^{10}T^5}$ \\
\hline
12 & $j = \dfrac{(T^2 + 4T - 8)^3(T^6 + 228T^5 - 408T^4 - 128T^3 - 192T^2 + 768T - 512)^3}{(T - 2)(T - 1)^3(T + 2)^3(T - 4)^{12}T^4}$ \\
\hline
13 & $\dfrac{(T^2 + 5T + 13)(T^4 + 247T^3 + 3380T^2 + 15379T + 28561)^3}{T^{13}}$ \\
\hline
16 & $j = \dfrac{(T^8 + 256T^7 + 5632T^6 + 53248T^5 + 282624T^4 + 917504T^3 + 1835008T^2 + 2097152T + 1048576)^3}{(T + 2)(T + 4)^4(T^2+4T+8)T^{16}}$ \\
\hline
18 & $j = \dfrac{(T^3 + 3T^2 - 9T + 9)^3(T^9 + 225T^8 - 1080T^7 + 3348T^6 - 8262T^5 + 16038T^4 - 23328T^3 + 26244T^2 - 19683T + 6561)^3}{(T - 1)^2T^9(T - 3)^{18}(T^2 - 3T + 3)(T^2 + 3)^2}$ \\
\hline
25 & $j = \dfrac{(T^{10} + 250T^9 + 4375T^8 + 35000T^7 + 178125T^6 + 631250T^5 + 1640625T^4 + 3125000T^3 + 4296875T^2 + 3906250T + 1953125)^3}{(T^4 + 5T^3 + 15T^2 + 25T + 25)T^{25}}$\\
\hline
\end{tabular}
\end{center}

\end{landscape}

\subsection*{Acknowledgments}

We are immensely indebted to John Duncan for his constant advice and encouragement throughout the project, without which this work would not have been possible. We are also grateful to Luca Candelori, Scott Carnahan, Simon Norton, Preston Wake, Robert Wilson and David Zureick-Brown for taking time to answer our questions, and to Ken Ono for posing the original problem which motivated this research. We thank the referees for the helpful suggestions and comments. We wrote part of this paper during our stay at the Erwin Schr\"{o}dinger International Institute for Mathematics and Physics (ESI) from 10 to 14 September 2018; we are thankful to ESI for their support.

\pdfbookmark[1]{References}{ref}
\LastPageEnding

\end{document}